\documentclass[11pt]{amsart}
\usepackage[top=3cm, bottom=2.5cm, left=2cm, right=2cm]{geometry} 

\usepackage{amsmath,amssymb,amsthm}
\usepackage[pdftex]{graphicx}
\usepackage{epsfig}
\usepackage[english]{babel}
\usepackage[usenames,dvipsnames]{color}  
\usepackage{xcolor}
\usepackage{enumerate}
\usepackage{float}

\newcommand{\R}{\mathbb{R}}

\newcommand{\hH}{\mathbb{H}}

\newcommand{\cH}{\mathcal{H}}
\newcommand{\cB}{\mathcal{B}}

\newcommand{\Sc}{{\rm Sc }}
\newcommand{\Ve}{{\rm Vec }}

\newcommand{\grad}{{\rm grad }}

\newcommand{\ov}{\overline}

\newcommand{\Cl}{{C \kern -0.1em \ell}}

\newcommand{\inner}[1]{\left\langle  #1 \right\rangle } 
\newcommand{\sinner}[1]{\left[  #1 \right] } 

\newcommand{\e}{{\bf e}}

\newcommand{\bu}{{\bf u}}
\newcommand{\bv}{{\bf v}}
\newcommand{\bw}{{\bf w}}
\newcommand{\bx}{{\bf x}}
\newcommand{\by}{{\bf y}}
\newcommand{\bB}{{\bf B}}
\newcommand{\bh}{{\bf h}}
\newcommand{\bH}{{\bf H}}
\newcommand{\bD}{{\bf D}}

\newtheorem{theorem}{Theorem}

\newtheorem{lemma}{Lemma}
\newtheorem{corollary}{Corollary}
\newtheorem{remark}{Remark}

\title[Variational principles in quaternionic analysis]{Variational principles in quaternionic analysis with applications to the stationary MHD equations}

\author{
P. Cerejeiras %\address{Departamento de Matem\'atica, Universidade de Aveiro, P 3810-193 Aveiro, Portugal. E-Mail: {\tt pceres@ua.pt}}
\and
U. K\"ahler % \address{Departamento de Matem\'atica, Universidade de Aveiro, P 3810-193 Aveiro, Portugal. E-Mail: {\tt ukaehler@ua.pt}}
\and
R.S.~Krau{\ss}har} % \address{Chair of Mathematics, Erziehungswissenschaftliche Fakult\"at, Universt\"at Erfurt, Nordh\"auser Str. 63, D-99089 Erfurt, Germany.  E-mail: {\tt soeren.krausshar@uni-erfurt.de}}}

\begin{document}
 
\begin{abstract}
In this paper we aim to combine tools from variational calculus with modern techniques from quaternionic analysis that involve Dirac type operators and related hypercomplex integral operators. The aim is to develop new methods for showing geometry independent explicit global existence and uniqueness criteria as well as new computational methods with special focus to the stationary incompressible viscous magnetohydrodynamic equations. 
 
We first show how to specifically apply variational calculus in the quaternionic setting. To this end we explain how the mountain pass theorem can be successfully applied to guarantee the existence of (weak) solutions. To achieve this, the quaternionic integral operator calculus serves as a key ingredient allowing us to apply Schauder's fixed point theorem. The advantage of the approach using Schauder's fixed point theorem is that it is also applicable to large data since it does not require any kind of contraction property. 

These consideration will allow us to provide explicit iterative algorithms for its numerical solution. Finally to obtain more precise a-priori estimates one can use in the situations dealing with small data the Banach fixed point theorem which then also grants the uniqueness.

\end{abstract}

\maketitle 

{\bf Keywords}: quaternionic integral operator calculus, stationary incompressible viscous magnetohydrodynamics equations, Dirac operators, existence and uniqueness theorems, variational calculus, mountain pass theorem, coercivity

{\bf MSC Classification}: 30 G 35; 76 W 05\\[0.2cm]

This paper is dedicated to Professor John Ryan on the occasion of his retirement.

\section{Introduction}
Nowadays, there are many powerful numerical methods available, for instance well-developed Finite Elements or Boundary Elements Methods, that can be used to solve numerically complicated non-linear systems of partial differential equations arising in modern problems of mathematical physics and engineering in dimension $n \ge 3$. According to our knowledge, much less has been developed on the level of analytical methods. To get deeper theoretical inside on the existence, uniqueness, regularity and the structure of solutions analytical methods are very useful, too. 

In 1989 K. G\"urlebeck and W. Spr\"o{\ss}ig proposed in their first book \cite{GS1} a new analytic toolkit for the treatment of three-dimensional non-linear elliptic boundary value problems, including the stationary viscous Navier-Stokes system. The main ingredients are quaternionic integral operators that arise from a higher dimensional generalization of complex function theory in the sense of the Riemann approach which considers null-solutions to the Euclidean Dirac operator, see also \cite{DSS} or the famous volume \cite{McIntosh} edited by J. Ryan which addressed  particularly applications to singular boundary value operators, including Calderon Zygmund type operators acting on Lipschitz surfaces. During the following two decades M. Shapiro, V.V. Kravchenko, S.  Bernstein, P. Cerejeiras, U. K\"ahler, J. Ryan, F. Sommen, N. Vieira, see for example \cite{Be,CK1,CK2,CV2009,GS1,K,KS,GKR}, and many others extended this new analytic machinery to also treat non-linear elliptic, parabolic and hypoelliptic systems. 

The quaternionic calculus leads to some new explicit criteria for the existence and the uniqueness of the solutions and also provided some information on the regularity behavior. Based on this quaternionic operator calculus also new numerical algorithms have been worked out in the first decade of this century, see for instance \cite{FGHK,GH}. In \cite{FGHK,GH} it was particularly shown how one can directly derive from the continuous quaternionic operator calculus a discrete quaternionic calculus version, which then even guarantees strong convergence instead of weak convergence only. Furthermore, for some particular domains under some special conditions these methods even allowed us to fully explicit representation formulas for the solutions to the Navier-Stokes equations, the Maxwell, Helmholtz and time independent Klein-Gordon equation, see for example \cite{ConKra5}. 

In the series of more recent papers \cite{CKK2019,KraussharTrends1,KraAACA2014} it has been shown how to apply the classical quaternionic operator calculus even to the much more complicated stationary viscous magnohydrodynamic equations (MHD) that combine the Navier-Stokes system with the Maxwell-equation. See also the more recent book \cite{GS3} Section~7.7.6.  

Since the works of M. Sermagne and R. Temam in 1983 the study of MHD equations have attracted the interested of a constantly and rapidly growing research community - just to mention a few milestones and some more recent developments we recommend the reader for instance to consult \cite{BBGHK,BD,Cannone,CMZ,DTE,Gala,GeSh2015,GP,GMP,GT,HW,Lei2015,Meir,MY,REM2017,ST,XZ-ZY2017} among many others.

Also complex quaternions and even octonions have been used for instance in \cite{DTE,TDT} in the description of the dynamics of dyonic plasmas in an elegant way. 

Most of the literature, including our recent papers \cite{CKK2019,KraussharTrends1,KraAACA2014}, in which we presented explicit estimates involving operator norms and the Reynolds numbers under which we obtain global existence and uniqueness of the solutions, use the Banach fixed point algorithm. 

We got an explicit geometry independent formula for the Lipschitz constant, however to have the contraction property we can only apply the presented algorithm to small data. This is a general problem when using the Banach fixed point algorithm, see also \cite{BKM}. One way to overcome this problem to also address large data is to work with Schauder's fixed point theorem instead, such as suggested for the Navier-Stokes system also in \cite{BKM}. To apply the ideas of Schauder's fixed point theorem in the framework of the quaternionic operator calculus however some substantial  new tools have to be developed first.

In this paper we now aim for the first time according to our knowledge to combine tools from variational calculus with tools from quaternionic analysis. Indeed, our approach then will allow us to develop new methods for showing geometry independent explicit global existence and uniqueness criteria as well as new computational methods using Schauder's fixed point theorem.  Again we concretely focus on the stationary incompressible viscous magnetohydrodynamic equations while the instationary case and other cases will be studied in some of our follow-up papers.  

Concretely, in Section~2.3 we show how to specifically apply the variational calculus in the quaternionic setting. To this end we explain in Section~2.4 how the mountain pass theorem can be successfully applied to guarantee the existence of (weak) solutions. To achieve this, the quaternionic integral operator calculus serves again as a key ingredient allowing us later in Section~3 to apply Schauder's fixed point theorem. As already pointed out, the great advantage of the approach using Schauder's fixed point theorem is that it is also applicable to large data since it does not require any kind of contraction property. These considerations will allow us to provide explicit iterative algorithms for its numerical solution which we also present in Section 3, see equations (\ref{u-modified})-(\ref{p-modified}).   
A disadvantage is of course that we only get the existence but in general not the uniqueness of a solution by following this approach. For the sake of giving a complete state of the art we also briefly summarize at the end of the paper in the appendix how the quaternionic operator calculus gives a unique solution when the Banach fixed point is applicable such that the reader is provided with an offrounded self-contained presentation.

%\red{Here we revisit some arguments in the context of more recent arguments of Ekeland... connecting our pervious work with more recent techniques from functional analysis (STILL TO DO!)}. 

%%%%%%%%%%%%%%%%%%%%%%%%
%%%%%%%%%%%%%%%%%%%%%%%%
\section{Toolkits for solving the stationary convective MHD equations}
\subsection{The stationary convective model}

We recall the  stationary convective case of the MHD equations. To leave it simple we work in the dimension-less setting and consider the system 
\begin{gather}  
-\frac{1}{R_e} \Delta {\bf u} + ({\bf u} \cdot \nabla) {\bf u} + \nabla p = \frac{1}{\mu_0} (\nabla \times {\bf B}) \times {\bf B}, \quad in ~ \Omega, \label{Eq:01} \\
-\frac{1}{R_m} \Delta {\bf B} - ({\bf u} \cdot \nabla) {\bf B} + ({\bf B} \cdot  \nabla) {\bf u}  = 0, \quad in ~ \Omega, \label{Eq:02} \\
\nabla \cdot {\bf u} =  \nabla \cdot {\bf B} = 0, \quad in ~ \Omega, \label{Eq:03} \\
{\bf u}=0, {\bf B}={\bf h}, \quad at ~\partial \Omega = \Gamma, \label{Eq:04}
\end{gather} where $\bu=(u_1, u_2, u_3)$ represents the velocity of the flow, $p$ the pressure, $\bB=(B_1, B_2, B_3)$ the magnetic field, $\mu_0$ is the magnetic permeability of the vacuum, and $R_e, R_m$ denote the fluid mechanical and the magnetic Reynolds numbers, respectively. 
 
Equations (\ref{Eq:03}) express the incompressibility of the flow and the non-existence of magnetic monopoles, respectively, while  (\ref{Eq:04}) indicate the data at the boundary of the domain. From now on, we shall refer to this system as System \ref{Eq:01}. 

This paper addresses the case where the non-linear convective terms $(\bu ~ \grad) \bu, (\bB ~\grad) \bu,$ and $(\bu ~\grad) \bB$ are not negligibly small, but where the flow is still viscous. % For simplicity sake we treat the case of ${\bf h}=0$ in detail. The more general situation can easily be adapted from our results by simply adding the inhomogeneous boundary term $F_{\Gamma} {\bf h} + T_G {\bf P}_G {\bf H}$ to the solution that we compute for ${\bf B}$.

We remark that all quantities are three-dimensional. Furthermore  $\cdot$ and $\times$ denote the standard inner product and vector product, respectively. Furthermore, we assume that $\Omega$ is a bounded domain in $\R^3$ with a boundary $\Gamma = \partial \Omega$ representable as a finite union of disjoint closed $C^2$ surfaces (each having finite surface area). More generally, our considerations can also be applied to the case where $\Gamma$ is a finite union of Liapunov surfaces or even a strongly Lipschitz surface in the sense of \cite{LMS,McIntosh}. 
%\item \red{some additional conditions on $\Gamma$? (CHECK!)}

\subsection{Quaternionic formulation of the problem} %Quaternions 

\subsubsection{Algebraic concepts}

We consider the quaternionic algebra 
\begin{equation} \label{Eq:2.01} \hH = \{ q= x_0 + x_1 \e_1 + x_2 \e_2 + x_3 \e_3 : \quad x_0, x_1, x_2, x_3 \in \R \}, 
\end{equation} where the elements $\e_1, \e_2, \e_3$ satisfy the  multiplication rules
\begin{equation}  \label{Eq:2.02} \e_1 \e_2 = - \e_2 \e_1 = \e_3, \quad \e_1^2 = \e_2^2 = \e_3^2 = -1. 
\end{equation} The term $x_0 =:\Sc(q)$ is called the real part or sometimes the scalar part of the quaternion $q$ and $\bx = x_1 \e_1 + x_2 \e_2 + x_3 \e_3$ is its vectorial part. The latter is also denoted by $\Ve(q)$. This consideration allows the embedding of $\R^3$ into $\hH$ by identifying a purely imaginary element $\bx = x_1 \e_1 + x_2 \e_2 + x_3 \e_3$ from $\hH$ with a vector $(x_1, x_2, x_3)$ from $\R^3$.  Since the product of two quaternions can be expressed in terms of a symmetric and an anti-symmetric part, we obtain the following relation between the symmetric and the anti-symmetric part, as well as between the standard inner product and the vector product, sometimes also called outer product, in $\R^3$:
\begin{equation}  \label{Eq:2.03} 
\bx \by = \frac{\bx \by + \by \bx}{2} + \frac{\bx \by - \by \bx}{2} = \Sc(\bx \by) + \Ve(\bx \by) = -\bx \cdot \by + \bx \times \by.
\end{equation}

The quaternionic conjugation is the involutory automorphism $\ov{\cdot} : \hH \rightarrow \hH$ defined by $q = x_0 +\bx \mapsto \ov{ q}= x_0 - \bx$. % and $\ov{q_1 q_2} = \ov{q_2}~ \ov{q_1},$ for all $q_1, q_2 \in \hH.$ 
Using this notation we can express the Euclidean norm in $\hH$ by $$| q |^2 = q\ov{q} = \ov q q = x_0^2 + x_1^2 + x_2^2 + x_3^2, \quad q= x_0 +\bx \in \hH.$$

%%%%%%%%%%%%%%%
\subsubsection{Quaternionic-Hilbert modules}

Given an open domain $\Omega \subset \R^3,$ we define a quaternion-valued function $f : \Omega \subset  \R^3 \rightarrow \hH$ as \begin{equation}  \label{Eq:2.04} 
\bx=x_1 \e_1 + x_2 \e_2 + x_3 \e_3 ~\mapsto ~f(\bx) =  f_0(\bx) + f_1(\bx) \e_2 + f_2(\bx) \e_2 + f_3(\bx) \e_3, \end{equation}where $f_0, f_1, f_2, f_3 : \Omega \subset  \R^3 \rightarrow \R$ are its real-valued coefficient functions. Properties such as continuity, etc.,  are to be understood component-wisely. A $C^1$ quaternion-valued function $f : \Omega \subset  \R^3 \rightarrow \hH$ is called monogenic in  $\Omega \subset \R^3$ if and only if it is a null solution of the Dirac operator $\bD,$ that is,
\begin{equation}  \label{Eq:2.05} 
\bD f =  (\partial_{x_1} \e_1 + \partial_{x_2} \e_2+\partial_{x_3} \e_3 ) f = 0, \quad {\rm in }~\Omega. \end{equation}
In the particular case of pure vectorial functions, i.e. functions with the mapping behavior $\bu = u_1 \e_1 + u_2 \e_2 + u_3 \e_3 : \Omega \subset  \R^3 \rightarrow \hH$ the action of $\bD$ can be expressed as 
\begin{equation}  \label{Eq:2.06} 
\bD \bu = \Sc(\bD \bu) + \Ve(\bD \bu) = -\nabla \cdot \bu + \nabla \times \bu.\end{equation} 
Additionally, we remark that $\Delta \bu = -\bD^2 \bu.$ \\

Using the quaternionic System (\ref{Eq:01})-(\ref{Eq:04}) can be rewritten in the form  
\begin{gather}
\frac{1}{R_e} {\bD}^2 {\bf u} - \Sc({\bf u} \bD) {\bf u} + \bD p = \frac{1}{\mu_0} \Ve ((\bD {\bf B})  {\bf B}), \quad in ~ \Omega, \label{Eq:05} \\
\frac{1}{R_m} {\bD}^2 {\bf B} + \Sc({\bf u} \bD)  {\bf B} - \Sc({\bf B} \bD)  {\bf u}  = 0, \quad in ~ \Omega, \label{Eq:06} \\
\Sc(\bD \bu) = \Sc( {\bD \bB}) = 0, \quad in ~ \Omega, \label{Eq:07} \\
{\bf u}=0, {\bf B}={\bf h}, \quad at ~\partial \Omega = \Gamma, \label{Eq:08}
\end{gather} 
which will be called System \ref{Eq:02}  from now on. Here again, $R_e$ and $R_m$ denote the mechanical and magnetic Reynolds numbers, respectively. 
Recall here that $ \Ve(\bD\bB) = \bD \bB$ since we always have that $\Sc(\bD \bB) = \nabla \cdot \bB =0$.

%%%%%%%%%%
Next we define a right (unitary) quaternionic function module as  a function space $V$ (over an open domain $\Omega \subset \R^3$ whose boundary is for instance a smooth Liapunov surface) together with the algebra morphism (also called right multiplication) $R : \hH \mapsto {\rm End}(V),$ given as the point-wise multiplication $\bx \mapsto (R(a)f)(\bx) = f(\bx)a.$ When $V$ is a Hilbert space we say we have a \textit{quaternionic Hilbert module}. \\

%A quaternionic Hilbert module $\mathcal{H}$ endowed with inner product $\inner{\cdot, \cdot}$ gives raise to the quaternionic valued functional
%\begin{equation} \label{HFunctional}
%\bv(\bu) = \inner{\bu, \bv} %= \int_\Omega \ov{\bu} \bv dx,
%\end{equation} for $\bu, \bv \in \mathcal{H}.$ Associated to this (quaternionic valued) functional is the scalar functional 
%\begin{equation} \label{sFunctional}
%\sinner{\bu, \bv} = \Sc(\inner{\bu, \bv}) %= \Sc \left( \int_\Omega \ov{\bu} \bv dx\right),
%\end{equation} 
%which corresponds to a real-valued inner product and, hence, gives rise to a norm in the classic sense.

The quaternionic Hilbert module $\mathcal{H}=L^2(\Omega; \mathbb H)$  gives rise to the quaternionic valued functional
\begin{equation} \label{HFunctional}
\inner{\bu, \bv}_{L^2} = \int_\Omega \ov{\bu} \bv \; d\bx,
\end{equation} for $\bu, \bv \in \mathcal{H},$ and where $d\bx= dx_1dx_2dx_3.$ Associated to this (quaternionic valued) functional is the scalar functional 
\begin{equation} \label{sFunctional}
\sinner{\bu, \bv} = \Sc(\inner{\bu, \bv}_{L^2}) = \Sc \left( \int_\Omega \ov{\bu} \bv \;d\bx \right),
\end{equation} 
which corresponds to a real-valued inner product and, hence, gives rise to a norm in the classic sense. Together with this norm and (quaternionic-valued) inner product the set $\cB(\Omega;\mathbb{H}) := L^2(\Omega;\mathbb{H}) \cap \{f: \Omega \to \mathbb{H} | ~{\bD} f(\bx) = 0,\;\forall \bx \in \Omega\}$ is a right-quaternionic Hilbert module called the quaternionic Bergman module of left monogenic functions. It is a reproducing kernel Hilbert module in the sense that there is a unique kernel $B(\bx,\by)$ satisfying $f(\bx) = \inner{B({\bf x},\cdot), f}$ for all $f \in \cB(\Omega;\mathbb{H})$. For details, see for instance \cite{DSS}. Associated to it one can consider the quaternionic Bergman projection $P:L^{2}(\Omega;\mathbb{H}) \to \cB(\Omega;\mathbb{H})$ given by 
$Pf(\bx) = \inner{B({\bf x},\cdot), f}$. This projection is an orthogonal projection with respect to $\inner{\cdot,\cdot}$. 

\par\medskip\par
In what follows, we consider 
$$\cH = H_2^1(\Omega; \hH) = \left\{ \bu \in L^2(\Omega; \hH) : \| \bu \|_{H_2^1} := \int_{\Omega} \left( |\bu |^2 + \sum_{i=1}^3 |\partial_{x_i} \bu  |^2 \right) d\bx < \infty \right\},$$ 
the Sobolev space of $L^2(\Omega; \hH)$ functions such that the functions and its weak derivatives have finite $L^2-$norm, together with the space of test functions 
$${\stackrel{\circ}{H_2^1}}(\Omega; \hH) = \left\{ \bu \in H_2^1(\Omega; \hH) : \left. \bu \right|_{\partial \Omega} = 0 \right\}.$$
Also for simplicity sake, we shall denote from now on the spaces $H_2^1$ and ${\stackrel{\circ}{H_2^1}}$ by $H^1$ and ${\stackrel{\circ}{H^1}}$, respectively, thus omitting the lower index $2$.
\par\medskip\par
Following \cite{GS1,GS2} and others, one has a Hodge type decomposition (in terms of a direct orthogonal sum) of the form 
$$
L^2(\Omega;\mathbb{H}) = \cB(\Omega;\mathbb{H}) \oplus {\bf D} {\stackrel{\circ}{H^1}}(\Omega;\mathbb{H}),
$$
where ${\stackrel{\circ}{H^1}}$ is the usual Sobolev space of the $L^2$-function having a first weak Sobolev derivative and with vanishing boundary data. In all that follows we shall denote the complementary part of the Bergman projection by $Q:= I -P$ where $I$ is the identity operator. Note that $Q:L^2(\Omega;\mathbb{H}) \to {\bf D} {\stackrel{\circ}{H^1}}(\Omega;\mathbb{H})$. 
\par\medskip\par
We now use the fact that the differential operator $\bD$ has a left (but not right) inverse. Following for instance \cite{GS1} and others for all $u \in C(\Omega;\mathbb{H})$ we have 
$$
(T_{\Omega} u)(\bx) := - \frac{1}{4\pi} \int_{\Omega} q_{\bf 0}(\bx-\by)u(\by) d\by,
$$
where $q_{\bf 0}(\bx) = -\frac{\bx}{|\bx|^3}$.  $T_{\Omega}$ is called the (quaternionic) Teodorescu transform. Its boundary analogue is the quaternionic Cauchy transform given by 
$$
(F_{\Gamma} u)(\bx) :=  \frac{1}{4\pi} \int_{\Gamma} q_{\bf 0}(\bx-\by)n(\by) u(\by) dS(\by)
$$
Here, $dS(\by)$ is the surface element and $n(\by)$ denotes the exterior normal unit at $\by \in \partial \Omega = \Gamma$. 

Again, we may recall from \cite{GS1} the following important properties
\begin{equation}\label{DT}
\forall u \in C^2(\Omega;\mathbb{H}) \cap C(\overline{\Omega};\mathbb{H}):\quad (\bD T_{\Omega} u)(\bx) = u(\bx) \quad \forall \bx \in \Omega
\end{equation}
\begin{equation}\label{TD}
\forall u \in C^1(\Omega;\mathbb{H}) \cap C(\overline{\Omega};\mathbb{H}):\quad (F_{\Gamma}u)(\bx) + (T_{\Omega}{\bD}u)(\bx) = u(\bx).
\end{equation}
For our needs it is important to mention that these two relations remain valid in $H^1(\Omega;\mathbb{H})$, cf. \cite{GS1}. 
\par\medskip\par
Now apply the operator $T_{\Omega} Q T_{\Omega}$ to System  \ref{Eq:02}. The properties (\ref{DT}), (\ref{TD}), application of the quaternionic Cauchy integral formula $QF_{\Gamma} {\bD} \bu = 0, QF_{\Gamma} p =  0$, $QF_{\Gamma} \bu = 0$ (since $u|_{\Gamma}=0$) and the property that $\bD \bu \in \Ve(Q)$ which implies that $T_{\Omega}Q{\bD} \bu = T_{\Omega} \bD \bu$ allow us to rewrite System \ref{Eq:02} in the following form, cf. for details \cite{KraussharTrends1,KraAACA2014}:  
\begin{eqnarray}
\bu &=& \frac{R_e}{\mu_0} T_{\Omega}QT_{\Omega} \Bigg[\Ve\big((\bD \bB)\bB \big)-\Sc(\bu\bD)\bu \Bigg]-R_e T_{\Omega}QT_{\Omega}\bD p \label{velocitynonlinear}\\
\Sc(Qp) &=& \frac{1}{\mu_0} \Sc\Bigg[  Q T_{\Omega}\Ve\big((\bD\bB)\bB\big) - \Sc(\bu\bD)\bu   \Bigg] \label{pressurenonlinear}\\
\bB &=& R_m T_{\Omega}QT_{\Omega}\Bigg[ \Sc(\bB\bD)\bu-\Sc(\bu\bD)\bB   \Bigg] + F_{\Gamma} \bh + F_{\Omega}P_{\Omega}{\bH}  \label{eqn_magnetic}\\
\bu &=& 0,\quad \bB= \bh\;\;{\rm at}\;\;\Gamma. 
\end{eqnarray}
This system will be called System \ref{Eq:03} in what follows. Also here and in all that is going to follow the symbols $R_e$ and $R_m$ stand for the dimension-less mechanical and magnetic Reynolds numbers, respectively. 
\par\medskip\par

These representation formulas, which are deduced by quaternionic function theory tools will play a crucial role our proof of existence where we want to use Schauder's fixed point theorem in order to also admit the case of large data, for its motivation see also \cite{BKM}.  When it is clear which domain $\Omega$ is considered, the indices $\Omega$ and $\Gamma$ will be omitted in the sequel for the sake of simplicity. 

\subsubsection{Preparatory norm estimates}

For our specific purposes we need some particular norm estimates involving some the quaternionic operators that we introduced before. 

To start, as an application of the identity ${\bD} T_{\Omega} f = f$ and the orthogonality $\langle \bu, \bD p\rangle_{L^2}$ 
we may directly obtain the following generalization of Lemma~4.1 from \cite{CK2}: 
\begin{lemma}
Let $\bu \in H_2^k(\Omega;\mathbb{H}) \cap {\rm ker}\;{\rm div}(\bD), k \ge 1, p \in L^2(\Omega;\mathbb{H}), \bB \in H_2^k(\Omega;\mathbb{H})$ be solutions of system \ref{Eq:02}. Then we have 
\begin{equation}\label{estim1}
\|\bD \bu\|_{L^2} + R_e \|Q T_{\Omega} \bD p\|_{L^2} \le R_e \|T_{\Omega} M(\bu)\|_{L^2} 
\end{equation}
where $M(\bu) := \Sc(\bu \bD) \bu - \frac{1}{\mu_0} \Ve((\bD \bB)\cdot \bB)$. 
\end{lemma}
We also require 
\begin{lemma}
Let $\Omega \subset \mathbb{R}^3$ be a domain with the general requirements mentioned at the end of Section~2.1. Then, relying on \cite{GS1},  we have the following estimate on the operator norm  
\begin{equation}\label{estim2}
\|T_{\Omega} Q T_{\Omega} \| \le C_1 \quad\quad with \;C_1 = \frac{1}{\lambda_{min}(-\Delta)},
\end{equation} 
where $\lambda_{\min}(-\Delta)$ denotes the minimal eigenvalue of the negative Laplacian $-\Delta$.  
\end{lemma}
\begin{proof}
This inequality follows from the norm estimate $\|T_{\Omega}\| = \|\Delta^{-1}\| \le \frac{1}{\sqrt{\lambda_{min}(-\Delta)}}$ and from the other norm identity $\|Q\|=1$ ($Q$ is an orthogonal projector). 
\end{proof}
Furthermore, applying analogous arguments as in \cite{CK2,GS1} we can infer 
\begin{lemma}(Norm estimates)\\
Let $\bu,\bB$ be a solution of System \ref{Eq:02}. Then there is a global constant $C_S > 0$ (called Poisson constant)   such that 
\begin{eqnarray*}
\|\Sc(\bu \bD) \bu \|_{L^q} & \le & C_S \|\bu\|_{H^1}^2\\
\|\Ve\big( (\bD\bB)\bB \big)\|_{L^q} & \le & C_S \|\bB\|_{H^1}^2 \\
\| \bD \bB\|_{L^2} &\le &  C_S \|\bB\|_{H^1},
\end{eqnarray*}
with $1<q<3/2$. 
\end{lemma}
\begin{proof}
For the first estimate we refer to~\cite{CK1}, Lemma 4.1. The second estimate is analogously to the first and the last is simply the boundedness of the Dirac operator from $H^1(\Omega)$ to $L_2(\Omega)$.
 
%Since $\|\cdot\|_{L^2}$ is a norm we have $\|\Sc(\bu\bD)\bu  \|_{L^2} \leq   \| \Sc(\bu \bD) \| \, \|\bu\|_{L^2} \leq  \| \bu \bD \| \, \|\bu\|_{L^2} $. 
%We know (see, e.g. \cite{GS1}) that  $\| \bu \bD \| \le C_S \|\bu\|_{L^2}$ involving the Sobolev embedding constant $C_2 > 0$. Hence, the inequality  $\|\Sc (\bu \bD) \bu \|_{L^2}  \le  C_2 \|\bu\|_{L^2}^2$ holds. The same constant $C_2$ satisfies the third equation, $\| \bD \bB\|_{L^2} \le  C_2 \|\bB\|_{L^2}$ as well as  $\|\Sc (\bB\bD) \bB  \| \le C_2 \|\bB\|_{L^2}^2$.  
%Since $\|\cdot\|_{L^2}$ is a norm we have $\|\Sc(\bu\bD)\bu\|_{L^2} = \|\Sc(\bu \bD)\| \, \|\bu\|_{L^2}$. 
%From \cite{GS1} and elsewhere we know that  $\|\Sc(\bu \bD)\| \le C_2 \|\bu\|_{L^2}$ involving the Sobolev embedding constant $C_2 > 0$. The same constant $C_2$ satisfies also $\|\Sc(\bB\bD)\| \le C_2 \|\bB\|_{L^2}$.  
%Additionally, we also have 
%\begin{eqnarray*}
%\|\Sc  (\bB\bD) \bu  \|_{L^2} &\le& C_2 \|\bB\|_{L^2} \, \|\bu\|_{L^2} \\
%\|\Sc (\bu\bD) \bB  \|_{L^2} &\le& C_2 \|\bB\|_{L^2} \, \|\bu\|_{L^2}. 
%\end{eqnarray*}
%Now, straightforward calculations show 
%\begin{eqnarray*}
%\|\Ve\big((\bD\bB)\bB \big)\|_{L^2} &=& \|\Sc(\bB\bD)\bB+\frac{1}{2} \bD |\bB|^2\|_{L^2} \\
%& \le & C_2 \|\bB\|^2_{L^2}+ \frac{C_2}{2} \|\bB\|^2_{L^2}\\
%&\le& \frac{3}{2}C_2 \|\bB\|^2_{L^2}.
%\end{eqnarray*}
%So the above inequalities hold for a global constant  $C_p \ge \frac{3}{2}C_2$.  
\end{proof}

\begin{remark}
Since the $T_\Omega$-operator is a bounded operator from $L_q(\Omega)$, $1<q<3/2$, to $L_2(\Omega)$ the above statement means that we have norm estimates like
$$
\|T_\Omega \Sc(\bu \bD) \bu \|_{L^2} \leq  C \|\bu\|_{H^1}^2.
$$
\end{remark}
 
\subsection{Variational formulation}
%%%%%%%
%PUTTING HERE THE INDICATION OF FUNCTION SPACES

Now we establish the variational formulation for System \ref{Eq:02}.   
The usual variational principles establish 
$$\inner{\partial_i \bu, \bv}_{L^2} = - \inner{\bu, \partial_i \bv}_{L^2}, \quad \bu \in H^1(\Omega; \hH), ~\bv \in {\stackrel{\circ}{H^1}}(\Omega; \hH), ~i=1,2,3.$$ In the case of the Dirac operator, we get 
$$\inner{\bD \bu, \bv}_{L^2} = \int_\Omega \ov{\bD \bu} ~\bv d\bx = \sum_{i,j, k} \int_\Omega ( \partial_{x_i} u_j) v_k \ov{\e_i \e_j} \e_k d\bx = - \sum_{i,j,k} \int_\Omega u_j \ov{\e_j} ( \partial_{x_i} v_k) \ov{\e_i} \e_k d\bx = \inner{\bu, \bD \bv}_{L^2},$$ as $\ov{\e_i} = -\e_i.$ Therefore, the Dirac operator is selfadjoint and has real eigenvalues.\\

%%%%%%%%%%%%%%%%

First of all, we observe that %we obtain for $\bv  \in H^1_0(\Omega),$ 
\begin{gather*}  
\frac{1}{R_e} \inner{\bD^2  {\bf u} , \bf v}_{L^2}  =  \frac{1}{R_e} \inner{\bD  {\bf u}, \bD \bf v}_{L^2}, %\\
% \inner{(\bu \cdot \bD) \bu, \bv} = \sum_{i,j,k} \left( \int_\Omega ( u_i \partial_{x_i} u_j ) v_k dx \right) \ov{\e}_j \e_k = - \sum_{i,j,k} \left( \int_\Omega   u_j (u_i \partial_{x_i} v_k) dx \right) \ov{\e}_j \e_k  = -\inner{\bu, (\bu \cdot \bD) \bv}, \\
%  \inner{\bD p, \bv} = \sum_{i,j} \left( \int_\Omega (\partial_{x_i} p) v_j dx \right) \ov{\e}_i \e_j = - \sum_{i,j} \left( \int_\Omega p (\partial_{x_i} v_j) dx \right) \ov{\e}_i \e_j =  \inner{p, \bD\bv}. 
  \end{gather*} %For the last term, we observe that 
%  $$ \bD \wedge \bB = \bD \bB$$ since $\nabla \cdot \bB \sim \bD \cdot \bB =0,$ so that
%\begin{gather*}   
% \frac{1}{\mu_0} \inner{(\bD \wedge {\bf B}) \wedge {\bf B}, \bv} =  \frac{1}{\mu_0} \inner{(\bD \bf B) \wedge {\bf B}, \bv} 
%\end{gather*}
so that Equation (\ref{Eq:05}) in System \ref{Eq:02} becomes      
\begin{equation} \label{Eq:2.08} 
\frac{1}{R_e} \inner{ {\bf D} {\bf u}, {\bf D} {\bf v}}_{L^2}   - \inner{\Sc(\bf u \bD) {\bf u} ,  {\bf v}}_{L^2} + \inner{\bD p, \bv}_{L^2} -\frac{1}{\mu_0} \inner{\Ve((\bD \bf B) \bf B), \bv}_{L^2} =0, \quad in ~\Omega. 
\end{equation}
In a similar way, we obtain for Equation (\ref{Eq:06}) the variational formulation 
\begin{equation} \label{Eq:2.09} 
\frac{1}{R_m} \inner{ {\bf D} {\bf B}, {\bf D} {\bf w}}_{L^2}  + \inner{\Sc(\bf u  \bD) {\bf B} - \Sc (\bf B \bD) {\bf u},  {\bf w}}_{L^2} =0, \quad in ~\Omega. 
\end{equation} Moreover, due to Equation (\ref{Eq:07}) one restricts the space of functions further as
$$H^{1, div}(\Omega; \hH) = \{ \bu \in H^1(\Omega; \hH) :  \Sc (\bD \bu) = 0  \}. $$

Hence, the variational form of System \ref{Eq:02} is the following: find $\bu, \bB \in H^{1,div}(\Omega; \hH)$, and $p \in L^2(\Omega; \hH),$ such that, for all $\bv, \bw \in {\stackrel{\circ}{H^1}}(\Omega; \hH)$ we have
\begin{equation} \label{Ref:System03} \left\{
\begin{array}{c} 
\frac{1}{R_e} \inner{ {\bf D} {\bf u}, {\bf D} {\bf v}}_{L^2}   - \inner{\Sc(\bf u \bD) {\bf u} ,  {\bf v}}_{L^2} + \inner{\bD p, \bv}_{L^2} -\frac{1}{\mu_0} \inner{\Ve((\bD \bf B) \bf B), \bv}_{L^2} =0, \quad in ~\Omega \\
~\\
\frac{1}{R_m} \inner{ {\bf D} {\bf B}, {\bf D} {\bf w}}_{L^2}  + \inner{\Sc(\bf u  \bD) {\bf B} - \Sc (\bf B \bD) {\bf u},  {\bf w}}_{L^2} =0, \quad in ~\Omega. 
\end{array} \right. \end{equation}

 \subsection{Energy functional} %We now address the problem of the energy of System \ref{Eq:02}. 
From the above variational formulation we obtain, for $\bv=\bu$ and $\bw =\bB,$ the energy functional 
%\begin{equation} \label{Ref:System04} \left\{\begin{array}{c} 
%\frac{1}{Re} \inner{ {\bf D} {\bf u}, {\bf D} {\bf u}}  - \inner{\Sc(\bf u \bD) {\bf u} ,  {\bf u}} + \inner{\bD p, \bu} -\frac{1}{\mu_0} \inner{\Ve((\bD \bf B) \bf B), \bu} =0, \quad in ~\Omega \\~\\
%\frac{1}{Rm} \inner{ {\bf D} {\bf B}, {\bf D} {\bf B}} + \inner{\Sc(\bf u  \bD) {\bf B} - \Sc (\bf B \bD) {\bf u},  {\bf B}} =0, \quad in ~\Omega. \end{array} \right. \end{equation} 

%Since the scalar functional (\ref{sFunctional}) is associated to the Euclidean norm we define the energy functional as the scalar part of the sum of the previous two equations with a Lagrange parameter:  
\begin{eqnarray}\label{Ref:System04}
J(\bu,\bB,p) & = & \frac{1}{R_e} \| {\bf D} {\bf u}\|^2_{L^2} - \Sc\left(\inner{\Sc(\bf u \bD) {\bf u} ,  {\bf u}}_{L^2} - \inner{\bD p, \bu}_{L^2} +\frac{1}{\mu_0} \inner{\Ve((\bD \bf B) \bf B), \bu}_{L^2}\right) \\
& & + \left(\frac{1}{R_m} \|  {\bf D} {\bf B}\|^2_{L^2}  + \Sc\left(\inner{\Sc(\bf u  \bD) {\bf B} - \Sc (\bf B \bD) {\bf u},  {\bf B}}_{L^2} \right)\right), \nonumber 
\end{eqnarray} with $\bu, \bB \in H^{1, div}(\Omega; \hH)$ and $p\in L^2(\Omega;\hH)$.\\
Notice that we are working in the dimensionless setting; otherwise one has to involve a constant correcting the physical units. 
Due to Equation (\ref{Eq:07}) we obtain 
$$\Sc(\inner{\bD p, \bu}_{L^2}) = \int_{\Omega} \sum_{i=1}^3 (\partial_{x_i} p ) u_i  d\bx = - \int_{\Omega} p \sum_{i=1}^3 (\partial_{x_i}  u_i ) d\bx =0,$$ since the divergence of $\bu$ is zero. %as $\nabla \cdot  \bu = \Sc(\bD \bu)=0.$ 
This implies that $\bD p$ is orthogonal to the velocity of the flow $\bu$ and, furthermore, the energy functional is independent on $p$. 

In a similar way, 
\begin{gather*} 
\Sc(\inner{\Sc(\bf u \bD) {\bf u} ,  {\bf u}}_{L^2})= \int_{\Omega} \sum_{i, j=1}^3 u_i (\partial_{x_i} u_j ) u_j  d\bx = \frac{1}{2}\int_{\Omega} \sum_{i, j=1}^3 u_i (\partial_{x_i} u_j^2 )  d\bx \\= - \frac{1}{2}\int_{\Omega} \sum_{j=1}^3 (\sum_{i=1}^3 \partial_{x_i} u_i ) u_j^2   d\bx =0,
\end{gather*} again due to the divergence of $\bu$ equal zero. This leads to $J(\bu,\bB,p) = J(\bu,\bB)$ and hence 
\begin{eqnarray}
J(\bu,\bB) & = & \frac{1}{R_e} \| {\bf D} {\bf u}\|^2_{L^2}-\frac{1}{\mu_0} \Sc\left( \inner{\Ve((\bD \bf B) \bf B), \bu}_{L^2} \right) \nonumber \\
& & +  \left(\frac{1}{R_m} \|  {\bf D} {\bf B}\|^2_{L^2}+ \Sc\left(\inner{\Sc(\bf u  \bD) {\bf B} - \Sc (\bf B \bD) {\bf u},  {\bf B}}_{L^2} \right) \right). \label{Eq:EnergyFunctional}
\end{eqnarray}

In the first place notice that $J$ is a lower semi-continuous function. It is also easy to see that $J$ is Fr\'echet differentiable. This allows us to apply the Mountain Pass Theorem which we recall here for convenience, cf. for example \cite{Zeidler}:

\begin{lemma}(Mountain Pass Theorem). 
Let $X$ be a Banach space, $J:X\mapsto\mathbb{R}$ a $C^1$-functional 
satisfying the Palais-Smale condition and $J(0)=0$. Suppose
\begin{enumerate}[a)]
\item there exists $\rho,\alpha>0$ such that $J|_{\partial B_\rho}\geq\alpha$ 
where $B_\rho$ denotes the ball of radius $\rho$ centered at $0$,
\item there exists $e\in X\setminus B_\rho$ with $J(e)\leq 0$.

Then $J$ has a critical value $c\geq \alpha$ and $c$ is given by 
$$
c=\int_{\Gamma} \max_{u\in\gamma([0,1])} J(u)
$$
where $\Gamma=\{\gamma\in C([0,1],X):\gamma(0)=0,\gamma(1)=e\}$.
\end{enumerate}
\end{lemma}
Using this Mountain Pass Theorem we are able to establish the main theorem of this section:
\begin{theorem}
The energy functional $J(\bu,\bB)$ has at least one minimum. 
\end{theorem}

\begin{proof}
Let us consider the energy functional $J(\bu,\bB)$ 

%in the form $$J(\bu,\bB,\lambda)=\Phi_1(\bu,\bB,\lambda)-\Phi_2(\bu,\bB,\lambda),$$
%with \begin{eqnarray*} \Phi_1(\bu,\bB,\lambda) & = & \frac{1}{Re} \| {\bf D} {\bf u}\|^2_{L^2} + \lambda \left(\frac{1}{Rm} \|  {\bf D} {\bf B}\|^2_{L^2}+ \Sc\left(\inner{\Sc(\bf u  \bD) {\bf B} - \Sc (\bf B \bD) {\bf u},  {\bf B}} \right) \right),\\ \Phi_2(\bu,\bB,\lambda) & = & \frac{1}{\mu_0} \Sc\left( \inner{\Ve((\bD \bf B) \bf B), \bu}\right).\end{eqnarray*}

Now consider $({\bf u},{\bf B})$. Again note that $p \in L^2(\Omega; \hH)$ has already been eliminated. The boundary of the ball of radius $\rho$ %condition $\|(\bu,\bB)\| = \rho$ 
thus is equivalent to 
$$
\|\bu\|^2_{H^1} + \|\bB\|^2_{H^1}  = \rho^2. 
$$
For subsequential purposes we note that %$\|\bB\|^2_{H^1} = \rho^2-\|\bu\|^2_{H^1}$, so 
we always can rely on 
\begin{equation}
\label{boundednessBbyrho}
\|\bB\|_{H^1} \le \rho, \quad \|\bu\|_{H^1} \le \rho.
\end{equation} 
To be able to apply the mountain pass theorem we have to show that 
$$
J(\bu,\bB) \ge \alpha \qquad \mbox{with} \quad \|\bu\|^2_{H^1} + \|\bB\|^2_{H^1}  = \rho^2.% \|(\bu,\bB)\|=\rho.
$$  
Now, equation (\ref{Eq:EnergyFunctional}) is equivalent to 
$$
J(\bu,\bB) = \frac{1}{R_e} \|{\bf D}\bu\|^2_{L^2} - \frac{1}{\mu_0} \Sc(\langle \Ve(({\bf D}\bB)\bB),\bu \rangle_{L^2}) 
+ \Bigg(        
\frac{1}{R_m} \|{\bf D}\bB\|^2_{L^2}
+\langle 
\Sc(\bu\bD)\bB-\Sc(\bB \bD)\bu,\bB
\rangle_{L^2}
\Bigg)
$$
%Working in the dimensionless setting, we can put without loss of generality $\lambda=1$. 
Furthermore, from the Poincar\'e inequalities $\|\bD\bu\|^2_{L^2} \le C_S \|\bu\|^2_{H^1}$ and $\|\bD \bB\|^2_{L^2} \le C_S \|\bB\|^2_{H^1}$ with  the same Poincar\'e constant $C_S$ which get the estimate
\begin{eqnarray*}
J(\bu,\bB) & \ge & \frac{1}{R_e} C_S \|\bu\|^2_{H^1} + \frac{1}{\mu_0} \|\bB\|^2_{H^1}  \|\bu\|_{H^1} + \frac{1}{R_m} C_S \|\bB\|^2_{H^1}- 2(\|\bu\|_{H^1} \|\bB\|^2_{H^1}) \\
& = & \frac{1}{R_e} C_S \|\bu\|^2_{H^1} + \frac{1}{R_m} C_S \|\bB\|^2_{H^1} - (2+\frac{1}{\mu_0})\|\bB\|^2_{H^1} \|\bu\|_{H^1}
\end{eqnarray*}
To show that 
\begin{equation}\label{estimateJ}
\frac{1}{R_e} C_S \|\bu\|^2_{H^1} + \frac{1}{R_m} C_S \|\bB\|^2_{H^1} - (2+\frac{1}{\mu_0})\|\bB\|^2_{H^1} \|\bu\|_{H^1} \ge \alpha
\end{equation}
we next use the identity 
$$
\|\bu\|_{H^1} \|\bB\|_{H^1} = \frac{1}{2} \Bigg(        
\|\bu\|^2_{H^1}+\|\bB\|^2_{H^1}-(\|\bu\|_{H^1}-\|\bB\|_{H^1})^2
\Bigg).
$$
Therefore we can estimate the expression on the left-hand side of inequality (\ref{estimateJ}) by 
\begin{eqnarray*}
& & \frac{1}{R_e} C_S \|\bu\|^2_{H^1}  + \frac{1}{R_m} C_S \|\bB\|^2_{H^1} 
- (2+\frac{1}{\mu_0}) \|\bB\|_{H^1} \cdot \Bigg(
\frac{1}{2} \{ \|\bu\|^2_{H^1}+ \|\bB\|^2_{H^1}-(\|\bu\|_{H^1}- \|\bB\|_{H^1})^2\}
\Bigg)\\
&\ge & \Bigg[ \frac{C_S}{R_e} - \frac{1}{2}(2+\frac{1}{\mu_0})\|\bB\|_{H^1} \Bigg] \|\bu\|^2_{H^1} 
+ \Bigg[  \frac{C_S}{R_m} - \frac{1}{2}(2 + \frac{1}{\mu_0}) \|\bB\|_{H^1}  \Bigg] \|\bB\|^2_{H^1}\\
&+& \underbrace{\frac{1}{2}(2+\frac{1}{\mu_0}) \|\bB\|_{H^1} (\|\bu\|_{H^1}-\|\bB\|_{H^1})^2}_{>0}.
\end{eqnarray*}  
The third term is positive and hence it can be dropped. Therefore, we finally arrived at 
\begin{equation}\label{estimatefinalJ}
J(\bu,\bB) \ge \Bigg[\frac{C_S}{R_e} - \frac{1}{2}(2 + \frac{1}{\mu_0})\|\bB\|_{H^1} \Bigg] \|\bu\|^2_{H^1} + 
\Bigg[\frac{C_S}{R_m} - \frac{1}{2}(2 + \frac{1}{\mu_0})\|\bB\|_{H^1} \Bigg] \|\bB\|^2_{H^1} 
\end{equation}
Using that $\|\bB\|_{H^1} \le \rho$ we obtain that 
$$
J(\bu,\bB) \ge \Bigg[\frac{C_S}{R_e} - \frac{1}{2}(2 + \frac{1}{\mu_0})\rho \Bigg] \|\bu\|^2_{H^1} + 
\Bigg[\frac{C_S}{R_m} - \frac{1}{2}(2 + \frac{1}{\mu_0}) \rho \Bigg] \|\bB\|^2_{H^1} 
$$
This means that for $\rho$ satisfying  $\rho < \frac{\min\{\frac{C_S}{R_e},\frac{C_S}{R_m}\}}{1+\frac{1}{2\mu_0}}$ we have 
$$
J(\bu,\bB) \ge \alpha(\rho) > 0.
$$
In fact, we here get a condition on how large the radius $\rho$ may be chosen. 

Choosing $(\bu_0,\bB_0)$ sufficiently far away from the origin we may find $J(\bu_0,\bB_0)<0$ and the conditions for the mountain pass theorem are fulfilled. This means that the functional $J(\bu,\bB)$ has at least one minimum. 
\end{proof}

We also get the following statement.  
\begin{corollary}
If 
$$
\|\bB\|_{H^1} < \frac{\min\{\frac{C_S}{R_e},\frac{C_S}{R_m}\}}{1+\frac{1}{2\mu_0}} 
$$
then $J(\bu,\bB)$ is coercive.  %and we could have got our result also from the Lemma of Lax-Milgram. 
\end{corollary} This corollary means that we could have got our result also from the Lemma of Lax-Milgram. 
This statement follows directly from (\ref{estimatefinalJ}) and it is a sufficient condition for coercivity. 
\begin{remark}
In the pure Navier-Stokes case (where $\bB={\bf 0}$) the energy functional reduces to 
$$
J(\bu) = \frac{1}{R_e} \|\bD \bu\|^2_{L^2}.
$$
In this case we always deal with a coercive system. We have $J(\bu)= 0$ if and only if $\bu \in {\rm ker}\bD$ and $J(\bu) > 0$ if and only if $\bD \bu \neq0$. The function theoretic property of being quaternionic monogenic is hence equivalent to the property that the energy functional $J$ vanishes. 
\end{remark}  

\section{Solvability via Schauder's fixed point theorem}
The previous calculations allowed us to study the solvability of the system in the form 
$$
J(\bu,\bB)\rightarrow\min.
$$
%A minimum must hence exist as a consequence of the mountain pass theorem as we proved in the preceding theorem. 

%%%%%%%%%%%%%%%%%%%%
%%%%%%%%%%%%%%%%%%%%
%%%%%%%%%%%%%%%%%%%%
%In this case, we consider 
%\begin{equation} \label{Eq:10} 
%\sinner{\bu, \bv} = \Sc \inner{\bu, \bv} = \Sc \left( \int_\Omega \ov{\bu} \bv dx \right). 
%\end{equation}
% 
%Consider now the trilinear form $b(\bu, \bv, \bw) = \sinner{(\bu \cdot \bD) \bv, \bw} =  \int_\Omega \sum_{i, j, k} u_i \frac{\partial v_j}{\partial x_i} w_j.$ In the particular case in which $\bv = \bw$ we get 
%\begin{gather*}
%b(\bu, \bv, \bv) = \sinner{(\bu \cdot \bD) \bv, \bv} =  \int_\Omega \sum_{i, j} u_i \frac{\partial v_j}{\partial x_i} v_j dx \\
%=  \frac{1}{2}\int_\Omega \sum_{i, j} u_i \frac{\partial}{\partial x_i} v^2_j dx 
%=  - \frac{1}{2}\int_\Omega \sum_{i, j} \frac{\partial u_i}{\partial x_i}  v_j^2  dx = 0,
%\end{gather*} as divergence of $\bu$ is null.
% 

%%%%%%%%%%%
%%%%%%%%%%%

Here, we can also use our variational formulation together with Schauder's fixed point theorem whose outcome will be our Theorem 2 presented at the end of this section representing the main result of this section. To this end we remark that our variational formulation can be considered in the form {\it to  find 
$({\bf u},{\bf B},p)$} such that 
\begin{eqnarray*}
%\int_\Omega \overline{\bD \bu}\; \bD \bv d\bx+\int_\Omega \overline{\mathrm{Sc}(\bu \bD) \bu}\; \bv +p (\bD \bv)-\frac{1}{\mu_0}\overline{[\mathrm{Vec}((\bD\bB)\cdot \bB]}\; \bv dx & = & 0\\
%\int_\Omega \overline{\bD\bB}\; \bD \bw  dx + \int_\Omega\overline{\mathrm{Sc}(\bu\bD)\bB}\;  \bw +\overline{\mathrm{Sc}(\bB\bD)\bu}\;  \bw dx & = & 0.
\int_\Omega \left( \overline{\bD \bu}\; \bD \bv +  \overline{\mathrm{Sc}(\bu \bD) \bu}\; \bv +p (\bD \bv)-\frac{1}{\mu_0}\overline{\mathrm{Vec}\big((\bD\bB)  \bB\big) }\; \bv \right) d \bx & = & 0\\
\int_\Omega \left( \overline{\bD\bB}\; \bD \bw  +  \overline{\mathrm{Sc}(\bu\bD)\bB}\;  \bw +\overline{\mathrm{Sc}(\bB\bD)\bu}\;  \bw \right) d \bx & = & 0.
\end{eqnarray*}
for all $\bv\in {\stackrel{\circ}{H^1}}(\Omega; \hH)$ and $\bw \in {\stackrel{\circ}{H^1}}(\Omega; \hH) \cap {\rm ker}\;{\rm div}$.  Note again that div$\bB = 0$ and div$\bw=0$. So we have $\bD\bB = {\rm rot} \bB$ and  $\bD\bw = {\rm rot} \bw$ in $\Omega.$ 

 Also we remark that   
$$\int_\Omega  p (\bD \bv) d \bx = \int_{\partial \Omega}  p {\boldsymbol n} \cdot \bv dS - \int_\Omega  (\bD p) \bv d \bx = 0$$
since $\bD p = $grad$(p)$ is orthogonal to $\bv.$ 

The linearized problem is now:  Given $(\tilde{\bu}, \tilde{\bB})$ we want to find $(\bu,\bB)$ such that
\begin{eqnarray}\label{MHDOseen1}
\int_\Omega \left( \overline{\bD \bu}\; \bD \bv +  \overline{\mathrm{Sc}(\tilde\bu \bD) \bu}\; \bv  \right) d \bx & = & \int_\Omega\frac{1}{\mu_0}\overline{\mathrm{Vec}\big((\bD\tilde \bB)\bB\big)}  \; \bv\; d \bx  \\ \label{MHDOseen2}
\int_\Omega \left( \overline{\bD\bB}\; \bD \bw  +  \overline{\mathrm{Sc}(\tilde \bu\bD)\bB}\;  \bw \right) d \bx & = & - \int_\Omega\overline{\mathrm{Sc}(\tilde \bB\bD)\bu}\;  \bw d \bx
\end{eqnarray}
for all $(\bv,\bw)$.

From~\cite{AML}, Lemma 2.1, we have
\begin{gather*}
\int_\Omega \overline{\mathrm{Sc}(\tilde\bu \bD) \bv}\; \bv\;  d \bx =  0, \quad 
\int_\Omega  \overline{\mathrm{Sc}(\tilde \bu\bD)\bw}\;  \bw\; d \bx = 0
\end{gather*}
such that the Lemma of Lax-Milgram gives us the unique solvability with the following norm estimates:
\begin{gather*}
\| \bu \|_{H^1}  \leq  \frac{1}{C_u}  \| \mathrm{Vec}\big( (\bD \tilde \bB)  \bB\big) \|_{H^{-1}}, \quad
%\| \bu \|_{H^1} & \leq & \frac{1}{C_u}  \| \mathrm{Vec}\big( (\bD \tilde \bB)  \bB\big) \|_{H^{-1}}, \\  %_{L_q} \\
\|\bB \|_{H^1}  \leq  \frac{1}{C_u} \|\mathrm{Sc}(\tilde \bB\bD)\bu\|_{H^{-1}}. 
%\|\bB \|_{H^1} & \leq & \frac{1}{C_u} \|\mathrm{Sc}(\tilde \bB\bD)\bu\|_{H^{-1}},  %_{L_q}
\end{gather*}
%due to the embedding $L_q$ into $H^{-1}$, 
where we have $C_u$ being the coercivity constant over the domain $\Omega$, i.e. 
$$
\left| \int_\Omega  \overline{\bD \bu}\; \bD \bu \; d\bx\right| \geq C_u \|\bu\|_{H^1}
$$ 

\par\medskip\par
Let us now consider the mapping $G: {\stackrel{\circ}{H^1}}(\Omega)\times H^1(\Omega)\to {\stackrel{\circ}{H^1}}(\Omega)\times H^1(\Omega)$, s.t $(\tilde{\bu},\tilde{\bB})\mapsto (\bu,\bB)$, 
which maps $(\tilde{\bu},\tilde{\bB})$ to the solution $(\bu,\bB)$ of 
(\ref{MHDOseen1})-(\ref{MHDOseen2}). 
\vspace{1cm}

 It is easily seen that $G$ is a continuous mapping. Schauder's fixed point theorem ensures the existence of a fixed point  $(\tilde{\bu},\tilde{\bB})= (\bu,\bB)$ if we guarantee that taken $(\tilde{\bu},\tilde{\bB})$ from a bounded convex set $\mathbb K$ we have that $G(\mathbb K)$ is compact. The map's continuity ensures the boundedness of $G(\mathbb K)$ while proving $({\bf u}, {\bf B}) \in H^{1+\epsilon}, ~\epsilon >0,$ implies compactness of $G(\mathbb K)$ due to the compact embedding of $H^{1+\epsilon}$ in $H^{1}.$
 
 \par\medskip\par
To prove that such a bounded convex set $\mathbb K$ does exist we need the equations in the $TQT$-form, but the fixed point iteration is constructed as a modification and in a different way   
from that considered in \cite{KraAACA2014}.
\begin{eqnarray}
\bu &=& \frac{R_e^2}{\mu_0} TQT[(\Ve(\bD \tilde{\bB})\bB) -\Sc(\tilde{\bu}\bD) \bu] - R_e^2 TQT 
\bD p \label{u-modified}\\
\bB &=& R_m^2 TQT[\Sc(\tilde{\bB}\bD)\bu - \Sc(\tilde{\bu}\bD)\bB] \label{B-modified}\\
\Sc(Qp) &=& \frac{1}{\mu_0} \Sc[QT(\Ve(\bD\tilde{\bB})\bB)-\Sc(\tilde{\bu}\bD)\bu]. \label{p-modified}
\end{eqnarray}
For simplicity we here assume that ${\bf h} = {\bf 0}$, otherwise one adds the term $F_{\partial \mathbb K} {\bf h} + T_{\mathbb K} {\bf P}_{\mathbb K} {\bf H}$ to the solution for ${\bf B},$ where ${\bf H} \in H^2(\mathbb K)$ is an extension of $\bf h$. 

Next,  the system above can be rewritten in the following way
\begin{eqnarray}
\bu &=& \big[I+\frac{R_e^2}{\mu_0} TQT\big(\Sc(\tilde{\bu}\bD)\big)\big]^{-1}R_e^2 TQT  \Big[\frac{1}{\mu_0} \Ve(\bD\tilde{\bB})\bB - {\bD}p \Big] \label{equreformulated}\\
\bB&=& \big[I+ R_m^2 TQT\big(\Sc(\tilde{\bu}\bD)\big)\big]^{-1} R_m^2 TQT \big(\Sc(\tilde{\bB}\bD) \big)\bu \label{eqbreformulated}. 
\end{eqnarray}
Equations (\ref{equreformulated}) and (\ref{eqbreformulated}) can be rewritten in form of Neumann series 
$$
\bu =  \Bigg(\sum\limits_{n=0}^{+\infty}\big[ -\frac{R_e^2}{\mu_0} TQT\big(\Sc(\tilde{\bu}\bD)\big) \big]^n \Bigg) R_e^2 TQT  \Big[ \frac{1}{\mu_0} \Ve(\bD\tilde{\bB})\bB - {\bD}p \Big],
$$ whenever $\left\| \frac{R_e^2}{\mu_0}  TQT \Sc(\tilde{\bu}\bD) \right\| = q_1 < 1,$ and 
$$
\bB = \Bigg(\sum\limits_{n=0}^{+\infty}\big[-R_m^2 TQT \big( \Sc(\tilde{\bu}\bD) \big) \big]^n  \Bigg) R_m^2 TQT \big( \Sc(\tilde{\bB}\bD \big) \bu,
$$
whenever $\left\| R_m^2  TQT \Sc(\tilde{\bu}\bD) \right\| = q_2 < 1.$ 

Further let us put $ k:= \|TQT\| \le \frac{1}{\lambda_{min}(-\Delta)}$. 

Then these two conditions are satisfied if $\|\tilde{\bu}\|_{H^1} \le \min\{  
\mu/(Re^2 k  C_D), 1/(R_m^2 k C_D),
\}$ %where $C_T$ is the operator norm of the $TQT$-operator 
where $C_D$ is the operator 
norm of the Dirac operator.

From these representations we may directly infer that 
\begin{eqnarray*}
\|\bu\|_{H^1} & \le &  \sum\limits_{n=0}^{+\infty} \left\| \frac{R_e^2}{\mu_0} TQT \big(\Sc(\tilde{\bu}\bD)\big) \right\|^n \left( \left\|  \frac{R_e^2}{\mu_0} TQT   \Ve(\bD\tilde{\bB}) \right\| \|\bB\|_{H^1} + \left\|  R_e^2 TQT  \right\| \| {\bD}p \|_{H^1}  \right) \\ 
& \le &  \frac{\left\|  \frac{R_e^2}{\mu_0} TQT   \Ve(\bD\tilde{\bB}) \right\| }{1-\left\| \frac{R_e^2}{\mu_0} TQT \big(\Sc(\tilde{\bu}\bD)\big) \right\| }  \|\bB\|_{H^1} + \frac{ \|  R_e^2 TQT \|}{1-\left\| \frac{R_e^2}{\mu_0} TQT \big(\Sc(\tilde{\bu}\bD)\big) \right\| }  \| {\bD}p \|_{H^1}  
\end{eqnarray*}
while
\begin{eqnarray*}
\|\bB\|_{H^1} & \le & \sum\limits_{n=0}^{+\infty} \|R_m^2 TQT \Sc(\tilde{\bu} \bD)\|^n \|R_m^2 TQT \Sc(\tilde{\bB} \bD)\| \|\bu\|_{H^1} \\
 &=& \frac{\|R_m^2 TQT \Sc(\tilde{\bB} \bD)\|}{1-\|R_m^2 TQT \Sc(\tilde{\bu}\bD)\|} \|\bu\|_{H^1}
\end{eqnarray*}
Next  we can estimate 
$$
\|\bu\|_{H^1} \le \frac{R_e^2 k \|\Ve(\bD\tilde{\bB})\|}{\mu_0-R_e^2 k \|\Sc(\tilde{\bu} {\bD})  \|} \|\bB\|_{H^1} + \frac{\mu_0 R_e^2 k  }{\mu_0-R_e^2 k \|\Sc(\tilde{\bu}{\bD})  \|} \|\bD p\|_{H^1}, 
$$ and 
$$\|\bB\|_{H^1} \le \frac{R_m^2 k \|\Sc(\tilde{\bB}\bD)\|}{1-R_m^2 k \|\Sc(\tilde{\bu} {\bD})  \|} \|\bu\|_{H^1}.
$$
The latter equation actually implies that there is a constant $C > 0$ such that 
$$
\|\bB\|_{H^1} \le C(\|\tilde{\bu}\|_{H^1},\|\tilde{\bB}\|_{H^1}) \|\bu\|_{H^1}.
$$
%Note that also equation (\ref{B-modified})  
%$$
%\bB= R_m^2 TQT [\Sc(\tilde{\bB}\bD)\bu- \Sc(\tilde{\bu}\bD)\bB]
%$$
%can also be reformulated in terms of a Neumann series 
%$$
%[I+R_m^2 TQT \Sc(\tilde{\bu}\bD))]^{-1} \bB = \sum\limits_{n=0}^{\infty} (-R_m^2 TQT \Sc (\tilde{\bu}\bD))^n  \bB
%$$
%and therefore 
%$$
%\|[I+R_m^2 TQT \Sc(\tilde{\bu}\bD))]^{-1}\| \le \sum\limits_{n=0}^{\infty} \|R_m^2 TQT \Sc(\tilde{\bu}\bD)\|^n \le \frac{1}{1-\|R_m^2 TQT 
%\Sc(\tilde{\bu}\bD)\|}.
%$$
%Since on the other hand we have
%$$
%\left\|\frac{R_e^2}{\mu_0} TQT \Sc(\tilde{\bu}\bD) \right\| = q_1 < 1 
%$$  
%then 
%$$
%\big[I + \frac{R_e^2}{\mu_0} TQT (\Sc(\tilde{\bu}\bD)) \big]^{-1} = \sum\limits_{n=0}^{\infty} (-\frac{R_e^2}{\mu_0} TQT \Sc(\tilde{\bu}\bD))^n
%$$
%and therefore 
%$$
%\|I + \frac{R_e^2}{\mu_0} TQT (\Sc(\tilde{\bu}\bD))^{-1}\| = \sum\limits_{n=0}^{\infty} \|\frac{R_e^2}{\mu_0} TQT \Sc(\tilde{\bu}\bD)\|^n = \frac{1}{1-C_1C_p \|\tilde{\bu}\|_{H^1}}.
%$$
Since we have 
$$\|  \Ve(\bD\tilde{\bB})  \bB  \|_{L^q} < +\infty, \qquad \mbox{\rm for }1< q < 3/2, $$ 
from Theorem 6.1. in \cite{BH} we have $\tilde{\bB} \in H^{1+\epsilon}$, $\epsilon > 0$.  
This ensures the necessary embedding of $(\bu, \bB) \in H^{1+\epsilon},$ for small enough $\epsilon >0,$ and concludes the compactness argument for $G(\mathbb{K}).$

We just have proved our main result of this section: 

\begin{theorem}
Let $\|\tilde{\bu}\|_{H^1} \le \min\{  
\mu/(Re^2 k C_D), 1/(R_m^2 k C_D)
\}$ then problem  (\ref{equreformulated}) and (\ref{eqbreformulated}) has a solution.
\end{theorem}
We remark that the condition of the theorem also implies the boundedness of 
$\tilde{\bB}$ although the theory does not provide uniqueness. To get also 
uniqueness we have to apply stronger fixed point theorems such as the 
Banach fixed point theorem which is summarized in the following appendix providing an 
offrounded presentation.

%Again we use the fact that  $\|TQT\| \le  \frac{1}{\lambda_{min}(-\Delta)}$. We have thus shown the boundedness property. In view of the continuity we may conclude that $G(\mathbb K)$ is compact and the theorem is hereby proven. 
%\red{Please check if our argumentation is complete! MENTION AND ARGUE WITH THE PAPER [4]! Please verify once more the calculations. It seems that the solutions $(\bB, \bu)$ are bnd if $\bu$ is (or if we start from a $\tilde{\bu}$ bnd - which we do). Formulate the main result in the form of a theorem, mention the concrete conditions on $\tilde{u}$ and $\tilde{B}$ to gurantee that it works }

\section{Appendix: Application of Banach's fixed point algorithm}
One advantage of using Schauder's fixed point theorem is that it can be applied to large data. A disadvantage consists in the fact that it usually does not guarantee the uniqueness of a solution. To achieve this one needs to use for instance the Banach fixed point theorem, which however can only be applied to small data as a consequence of the contraction property that is required. In some preceding papers, see \cite{KraussharTrends1,KraAACA2014,CKK2019}, we already showed that the quaternionic operator calculus can also be successfully used in this context. For the sake of completeness we finish this paper by adding a concise summary on the application of Banach's fixed point theorem on the level of these quaternionic integral operators so that the reader gets a rather complete picture of the current state of the art.

%\red{Additionally, we successfully revisit this algorithm in the context of applying the more modern techniques by Ekeland in our argumentation --- this we still have to do}.   

To start we wish to emphasize that the representation formulas for the velocity $\bu$, the pressure $p$ and the magnetic field $\bB$ presented in System \ref{Eq:03} (19)-(22) also hold when the convective terms are high in comparison with the viscous terms. In the case where the viscous terms dominate the convective terms (the exact conditions will be given Theorem~3 and Theorem~4) we may use the following Banach fixed point algorithm to compute ${\bf u}$, ${\bf B}$ and $p$, see \cite{KraAACA2014}: 

\begin{eqnarray*}
{\bf u}_n &=& \frac{R_e}{\mu_0} T_G {\bf Q} T_G \Big[\Ve(({\bf D}{\bf B}_{n-1}) {\bf B}_{n-1})-\Sc({\bf u}_{n-1}{\bf D}){\bf u}_{n-1}\Big] - R_e^2 T_G {\bf Q} T_G {\bD}   p_n\\
\Sc({\bf Q}  p_n)&=&\frac{1}{\mu_0} R_e\Big[{\bf Q} T_G \Ve(({\bf D}{\bf B}_{n-1}) {\bf B}_{n-1}) - \Sc({\bf u}_{n-1}{\bf D}){\bf u}_{n-1}\Big]\\
{\bf B}_n &=& R_m^2 T_G {\bf Q} T_G \Big[\Sc({\bf B}_n{\bf D}){\bf u}_n - \Sc({\bf u}_n{\bf D}){\bf B}_n \Big],
\end{eqnarray*}where $n=1,2,\ldots$

\begin{remark}
Note that whenever the equations (\ref{velocitynonlinear}) and (\ref{pressurenonlinear}) have a unique solution, then also (\ref{eqn_magnetic}) will have a unique solution. 
\end{remark}
 
In practice we can first compute ${\bf B}_n$ by applying the inner iteration:
\begin{equation}\label{inner}
{{\bf B}_n}^{(i)} = R_m^2 T_G {\bf Q} T_G \Big[\Sc({{\bf B}_n}^{(i-1)}{\bf D}){\bf u}_n - \Sc({\bf u}_n{\bf D}){{\bf B}_n}^{(i-1)} \Big],\;\;i=1,2,\ldots
\end{equation}
In \cite{KraAACA2014} we  proved 
\begin{theorem}
Suppose that ${\bf u}_n \in {\stackrel{\circ}{H^1}}(G)$ and that $\bu$ fulfils the condition  
\begin{equation}\label{cond1}
\|{\bf u}_n\|_{H^1} < \frac{1}{2 C_1 C_S R_m^2}.
\end{equation}
Then the problem 
\begin{equation}\label{probl}
{\bf B}_n = R_m^2 T_G {\bf Q} T_G \Big[\Re({\bf B}_n{\bf D}){\bf u}_n - \Re({\bf u}_n{\bf D}){\bf B}_n \Big]
\end{equation}  
has a uniquely solution in $H^{1}(G)$ . The sequence defined in~(\ref{inner}) converges in $H^{1}(G)$ to a unique solution of (\ref{probl}). 
\end{theorem}
{\bf Sketch of the proof} (for all the details, see \cite{KraAACA2014}). 
Note that to show the convergence of the inner iteration proposed in (\ref{inner}), 
we can directly use the estimates from Lemma 2 and Lemma 3. From the simple fact that $\Sc(a\pm b)=\Sc(a)\pm \Sc(b)$, we may infer from Lemma 3 that  
$$
\|\Sc({\bf B}_n^{(i-1)}{\bf D})-\Sc({\bf B}_n^{(i-2)}{\bf D})\|_{H^1} \le C_S \|{\bf B}_n^{(i-1)}-{\bf B}_n^{(i-2)}\|_{H^1},
$$ 
involving again the same positive constant $C_S$. The norm estimates of Lemma 2 and Lemma 3 yield   
\begin{eqnarray*}
\|{\bf B}_n^{(i)} - {\bf B}_n^{(i-1)}\|_{H^1} &=&R_m^2 \|T_G {\bf Q} T_G [\Re({\bf B}_n^{(i-1)}{\bf D}){\bf u}_n - \Re({\bf u}_n {\bf D}) {\bf B}_n^{(i-1)}] \|_{H^1}\\
& \le & R_m^2 C_1 \|\Re(({\bf B}_n^{(i-1)}-{\bf B}_{n}^{(i-2)}){\bf D})\| \|{\bf u}_n\|_{H^1} %\\
%& + & 
+R_m^2 C_1 \|\Re({\bf u}_n {\bf D})({\bf B}_n^{(i-1)}-{\bf B}_n^{(i-2)})\|_{H^1}\\
& \le & 2 R_m^2 C_1 C_S \|{\bf B}_n^{(i-1)}-{\bf B}_n^{(i-2)}\|_{H^1} \|{\bf u}_n\|_{H^1}.
\end{eqnarray*}
The classical Banach's fixed point theorem now tells us that ${\bf B}_n^{(i)}$ $i=1,2,\ldots$ converges under the condition (\ref{cond1}) 
to a unique solution for $\bB_n$. Finally, since each ${\bf u}_n$ is an element of ${\stackrel{\circ}{H^1}}(G)$, it is bounded over $G$. 
 \hfill $\blacksquare$

\par\medskip\par

Now one can further estimate using Lemma 2 and Lemma 3 that 
$$
\|\bB_n - \bB_{n-1}\|_{H^1} \le R_m^2 C_1 F \|\bu_n-\bu_{n-1}\|_{H^1}
$$
with $F:= 2C_S(\|\bB_{n}\|_{H^1}+\|\bB_{n-1}\|_{H^1}) \le 4 C_S \sup\limits_{n \in \mathbb{N}}\{\|\bB_n\|_{H^1}\}\}$. For details see again \cite{KraAACA2014}. The problem is thus reduced to show the convergence of $(\bu_n)_{n \in \mathbb{N}}$.

\par\medskip\par

Following further \cite{KraAACA2014} one can estimate 
$$
\|{\bf u}_n-{\bf u}_{n-1}\|_{H^1} \le  2 R_e^2 C_1  \|M({\bf u}_{n-1}) - M({\bf u}_{n-2})\|_{H^1}.
$$
Furthermore, with $C_3:=\|{\bf u}_{n-1}\|_{H^1}+\|{\bf u}_{n-2}\|_{H^1}$ and $C_4:=\|{\bf B}_{n-1}\|_{H^1}+\|{\bf B}_{n-2}\|_{H^1} \le 2 \sup_{n \in {\mathbb{N}}}\{\|{\bf B}_n\|_{H^1}\}$ we have 
\begin{eqnarray*}
  \|M({\bf u}_{n-1}) - M({\bf u}_{n-2})\|_{H^1}
 &\le & C_S C_3 \|{\bf u}_{n-1} - {\bf u}_{n-2}\|_{H^1}\\
 &&\quad + \frac{C_S C_4}{\mu_0}R_m^2 C_1 F  \|{\bf u}_{n-1} - {\bf u}_{n-2}\|_{H^1}\\
 &&\quad + \frac{C_4}{2 \mu_0} R_m^2 C_1 F \|{\bf u}_{n-1} - {\bf u}_{n-2}\|_{H^1} 
\end{eqnarray*}
This leads to the estimate  
\begin{equation}
\|{\bf u}_n-{\bf u}_{n-1}\|_{H^1} \le  L_n \|{\bf u}_{n-1}-{\bf u}_{n-2}\|_{H^1}
\end{equation}
with the Lipschitz constant
\begin{equation}
\label{Ln}
L_n:=2 R_e^2 C_1 \Big(C_S C_3 + \frac{(1/2+C_S)C_4}{\mu_0}R_m^2 C_1 F \Big)
\end{equation}

Banach's fixed point theorem tells us that we get convergence if $L_n < 1$. 

\par\medskip\par
We can say more. 

\begin{theorem}(cf. \cite{KraAACA2014}) Under the conditions 
$$
\frac{1}{\mu_0} \sup_{n \in \mathbb{N}}\|{\bf B}_n\|_{H^1}^2 \le \frac{1}{16 C_1^2 C_S^2 R_e^4}
$$
and 
$$
R_m^2 < \frac{4C_S^2R_e^2(8W R_e^2 C_1 C_S-1)}{1+2C_S},
$$ 
where
$$
W:=\sqrt{\frac{1}{4C_1^2C_S^2 R_e^4}-\frac{1}{\mu_0}\sup_{n \in \mathbb{N}}\|{\bf B}_n\|_{H^1}^2}
$$
the proposed fixed point algorithm converges to a unique triple of solutions $({\bf u},p,{\bf B})$ of the boundary value problem (11)-(15), where $p$ is unique up to a constant when $G$ is bounded, and fully unique when $G$ is unbounded.    
\end{theorem}
For the detailed proof we refer  \cite{KraAACA2014}. 
\par\medskip\par
{\bf Remark}: The interested reader may also consult \cite{GS3} Section~7.7.6 where the authors also followed basically the same line of arguments using a slightly different notation but established rather analogous estimates based on similar norm constants. The results are qualitatively consistent with those that were obtained previously in \cite{KraAACA2014}. 
 
\par\medskip\par
{\bf Remarks}: This method provides us with an explicit estimate on the Lipschitz contraction constant giving an a priori estimate on how many iterations are needed, cf. \cite{GS1,GS2}.  However, since the Lipschitz constant must be smaller than $1$, it can only be applied to small data. A further advantage of the quaternionic operator calculus method is that it always guarantees strong convergence when switching to the discretized versions of these quaternionic operators. It is worthwhile to mention that in many (practical) situations the Teodorescu transform can also be replaced by a simpler primitivation operator \cite{BGSS} whose evaluation requires much less computational steps. This is also a significant input and added value of the quaternionic function theory, since these fnction theoretical concept allow the replacement of a 3D-integral by a one-dimensional primitivation under certain circumstances. Details and extensions to the instationary case will be treated in one of our follow up papers.  

%\red{I guess we should also cite some things from Zeidler - but where precisely? Please check literature list and references}

\section*{Acknowledgments}
P. Cerejeiras and U. K\"ahler were supported by Portuguese funds through the CIDMA - Center for Research and Development in Mathematics and Applications, and the Portuguese Foundation for Science and Technology (``FCT--Funda\c{c}\~ao para a Ci\^encia e a Tecnologia''), within project UIDB/04106/2020 and UIDP/04106/2020. They also acknowledge the support of the NSFC project ref: NSFC11971178.

\vspace{1cm}
AFFILIATIONS:\\[0.2cm] 
P. Cerejeiras: Departamento de Matem\'atica, Universidade de Aveiro, P 3810-193 Aveiro, Portugal. E-Mail: {\tt pceres@ua.pt}\\[0.2cm] 
U. K\"ahler: Departamento de Matem\'atica, Universidade de Aveiro, P 3810-193 Aveiro, Portugal. E-Mail: {\tt ukaehler@ua.pt}\\[0.2cm] 
R.S. Krau{\ss}har: Chair of Mathematics, Erziehungswissenschaftliche Fakult\"at, Universt\"at Erfurt, Nordh\"auser Str. 63, D-99089 Erfurt, Germany.  E-mail: {\tt soeren.krausshar@uni-erfurt.de

\end{document}